\documentclass{amsart}
\usepackage{amsmath,amsthm,amsopn,amsfonts,amssymb}

\newtheorem{thm}{Theorem}[section]

\newtheorem{lemma}[thm]{Lemma}

\theoremstyle{definition}
\newtheorem{defn}[thm]{Definition}
\theoremstyle{remark}
\newtheorem{rem}[thm]{Remark}

\numberwithin{equation}{section}

\newcommand{\norm}[1]{\left\Vert#1\right\Vert}

\newcommand{\N}{\mathbb{N}}
\newcommand{\R}{\mathbb{R}}

\newcommand{\diverg}{\textnormal{div}}
\newcommand{\dys}{\displaystyle}

\begin{document}
\title[The regularizing effects of the some order terms...]{The regularizing effects of some lower order terms in an elliptic equation with degenerate coercivity}

\author{Gisella Croce}
\address{Laboratoire de Math\'ematiques Appliqu\'ees du Havre
\\
Universit\'e du Havre
\\
25, rue Philippe Lebon, BP 540, 76058 Le Havre FRANCE}
\email{gisella.croce@univ-lehavre.fr}
\thanks{}
\begin{abstract}
In this article we study an elliptic problem with degenerate coercivity. 
We will show that the presence of some lower order terms
has a regularizing effect on the solutions, if we assume that the datum $f \in L^m(\Omega)$ for some $m\geq 1$.
\end{abstract}

\maketitle

\section{Introduction}
Consider the problem
\begin{equation}\label{laplaciano_polinomio}
\left\{
\begin{array}{cl}
\displaystyle
-\Delta u + |u|^{p-1}u=f& {\rm in}\,\,\Omega
\\
u=0 & {\rm on}\,\,\partial \Omega\,.
\end{array}
\right.
\end{equation}
It is known that the lower order term $|u|^{p-1}u$ has a regularizing effect on the solutions: indeed  
the solution $u$ belongs to the Lebesgue space $L^p(\Omega)$ under the weak assumption that the datum $f \in L^1(\Omega)$ (see \cite{strauss}); moreover 
$\nabla u \in L^q(\Omega)$, $q<\frac{2p}{p+1}$ (see \cite{bgv}).

A stronger effect 
can be observed if we consider a lower order term $h(u)$ where $h: [0,\sigma)\to \R^+$ is a continuous, increasing function with a vertical asymptote in 
$\sigma\,\, (\sigma>0)$. 
Indeed in \cite{B_proceed} it is shown that if $f$ is an $L^1(\Omega)$ positive function, then there exists a bounded $H^1_0(\Omega)$ solution to
\begin{equation}\label{laplaciano_asintoto}
\left\{
\begin{array}{cl}
\displaystyle
-\Delta u + h(u)=f& {\rm in}\,\,\Omega
\\
u=0 & {\rm on}\,\,\partial \Omega\,.
\end{array}
\right.
\end{equation}
All these results have been proved for general nonlinear coercive elliptic problems too.

In this article we want to analyse the regularizing effects of the same kind of lower order terms as in problems (\ref{laplaciano_polinomio}) 
and (\ref{laplaciano_asintoto})
in the case of an elliptic operator with degenerate coercivity.  
More in details, let us consider the differential operator
\begin{equation}\label{operatore}
A(v)=-\diverg\left(a(x,v)\nabla v\right)\,,\,\,\,\,\,\,\,v \in H^1_0(\Omega)
\end{equation}
under the following assumptions: 
$\Omega$ is an open bounded subset of $\R^N, N\geq 3,$
and 
$a: \Omega \times \R\to \R$ is a Carath\'eodory function such that for a.e. $x \in \Omega$ and for every $s \in \R$
\begin{equation}\label{ellitticita}
a(x,s)\geq \frac{\alpha}{(1+|s|)^{\gamma}}
\end{equation} 
and 
\begin{equation}\label{crescita}
|a(x,s)|\leq \beta
\end{equation} 
for some real positive constants $\alpha, \beta, \gamma$. 
Assumption (\ref{ellitticita}) implies that the differential operator 
$A$ 
is well defined on $H^1_0(\Omega)$, but it fails to be coercive 
on the same space when $v$ is large (see \cite{Po}). 
Due to the lack of coercivity, the classical theory for elliptic operators acting 
between spaces in duality (see \cite{L}) cannot be
applied. 
Despite this difficulty, some papers (see \cite{abfot}, \cite{LB}, \cite{bb}, \cite{bdo}, 
\cite{cirmi-degenere}, \cite{michaela} and \cite{daniela}) have been written about 
Dirichlet problems with the differential operator $A$.
In particular in  \cite{abfot}, \cite{LB} and \cite{bdo} it has been shown the existence of solutions to the differential problem
\begin{equation}\label{without}
\left\{
\begin{array}{cl}
\displaystyle
-\diverg\left(a(x,u)\nabla u\right)=f& {\rm in}\,\,\Omega
\\
u=0 & {\rm on}\,\,\partial \Omega
\end{array}
\right.
\end{equation}
in the case where $f\in L^m(\Omega)$ with $m\geq 1$
and $\gamma \in (0,1]$. 
Indeed, let  
\begin{equation}\label{somm_u}
r=\frac{Nm(1-\gamma)}{N-2m}
\end{equation}
and 
\begin{equation}\label{somm_grad}
{q}=\frac{Nm(1-\gamma)}{N-m(1+\gamma)}\,.
\end{equation}
If $f \in L^m(\Omega)$, with $1\leq m\leq \max\left\{\frac{N}{N+1-\gamma(N-1)}, 1\right\}$
problem (\ref{without}) admits a solution $u$ such that
$$
|u|^s \in L^{1}(\Omega),\,\,\,\,s<r
$$
and 
$$
|\nabla u|^s \in L^{1}(\Omega),\,\,\,\,\,s<{q}\,.
$$
If $\frac{N}{N+1-\gamma(N-1)}<m \leq \frac{2N}{N(1-\gamma)+2(\gamma + 1)}$,
then $$u \in W^{1,q}_0(\Omega).$$
If $\frac{2N}{N(1-\gamma)+2(\gamma + 1)}\leq m <\frac{N}{2}$, then 
$$u \in H^1_0(\Omega)\cap L^{r}(\Omega)\,.$$ 
If $m>\frac{N}{2}$, then 
$$
u \in H^1_0(\Omega)\cap L^{\infty}(\Omega)\,.
$$

In the case where $\gamma>1$, the effect of the degenerate coercivity is even worst: 
indeed problem (\ref{without}) 
may have no solution at all, even if the datum $f$ is a constant function (see \cite{bdo} and \cite{abfot}).
In any case the degenerate coercivity of the operator $A$ gives solutions less regular than 
the solutions to uniformly elliptic problems.

In the spirit of problem (\ref{laplaciano_polinomio}), 
the first problem that we will consider in this article is  
\begin{equation}\label{problema_di_Dirichlet}
\left\{
\begin{array}{cl}
\displaystyle
-\diverg\left(a(x,u)\nabla u\right)+|u|^{p-1}u=f& {\rm in}\,\,\Omega
\\
u=0 & {\rm on}\,\,\partial \Omega\,.
\end{array}
\right.
\end{equation}
We will study the existence of distributional solutions, that is
the existence of some function $u$ in $W^{1,1}_0(\Omega)$ (at least) such that $|u|^p \in L^1(\Omega)$ and
$$
\int_{\Omega}{a(x,u)}{\nabla u}\cdot \nabla \varphi+
\int_{\Omega}|u|^{p-1}u\,\varphi=\int_{\Omega}f\,\varphi
$$
for every $\varphi \in C^{\infty}_0(\Omega).$
We will also use a weaker notion of solution, 
the entropy solution. We recall that this type of solution has been introduced in \cite{BBGGPV}.
To give this definition we need the following lemma (see \cite{BBGGPV} for the proof). We define
$$
T_k(s)=\max\{-k, \min\{k,s\}\}
$$
\begin{lemma}
Let  $u: \Omega \to \R$ be measurable function  such that $T_k(u) \in H^1_0(\Omega)$ for every $k>0$. Then  
there exists an unique measurable $v:\Omega\to \R^N$ such that
$$
\nabla T_k(u)=v\,\chi_{\{|u|<k\}}\quad {q.o.\,\, in}\,\, \Omega\,\,\,\forall\, k>0\,;
$$
the map $v$ is called the weak gradient of $u$. 
If $u \in W^{1,1}_0(\Omega)$, then the weak gradient of $u$ coincides with its standard distributional gradient. 
\end{lemma}
\begin{defn}\label{def-entropy}
Let $f$ be an $L^m(\Omega)$ function, with $m\geq 1$. 
A measurable function $u: \Omega \to \R$ such that 
$|u|^p \in L^1(\Omega)$ 
and $T_k(u) \in H^1_0(\Omega)$ for every $k>0$ 
is an entropy solution to  problem (\ref{problema_di_Dirichlet}) 
if
$$
\int_{\Omega}{a(x,u)\nabla u\cdot\nabla T_k(u-\varphi)}+
\int_{\Omega} |u|^{p-1}u \,T_k(u-\varphi)
\leq 
\int_{\Omega}{f\,T_k(u-\varphi)}
$$
for every $k>0$ and for every $\varphi\in H^{1}_0(\Omega)\cap L^{\infty}(\Omega).$
\end{defn}
Note that this definition is useful in the case where the solution of 
problem (\ref{problema_di_Dirichlet}) does not necessary belong to a Sobolev space. 
Indeed, about the gradient of the solution,  it has a sense under the weak hypothesis that
$\nabla T_k(u) \in L^2(\Omega)$, we don't need that $\nabla u \in L^1(\Omega)$, as for distributional solutions.

\begin{defn} 
The Marcinkiewicz space $M^s(\Omega)$, $s>0$, is made up of all measurable functions $v: \Omega \to \R$ such that
$|\{|v|\geq k\}|\leq \frac{C}{k^s}$ for all $k>0$, for a positive constant $C>0$.
\end{defn}
One can show that, if $|\Omega|$ is finite and if $0<\varepsilon<s-1$
\begin{equation}\label{embedding-sob-marc}
L^s(\Omega)\subset M^s(\Omega)\subset L^{s-\varepsilon}(\Omega)\,.
\end{equation}

\smallskip
We can now state our existence results for problem (\ref{problema_di_Dirichlet}).
\begin{thm}\label{dato_in_luno}
Let $f \in L^{1}(\Omega)$. 
\begin{enumerate}
\item
If $p>\gamma+1$, then there exists a distributional solution $u$ to problem (\ref{problema_di_Dirichlet}) such that 
$$
u \in W^{1,s}_0(\Omega)\cap L^p(\Omega)\,,\,\,\,s<\frac{2p}{\gamma +1+p}\,.
$$ 
\item 
If $0<p\leq \gamma +1$, then there exists an entropy solution $u$ to problem (\ref{problema_di_Dirichlet}) such that
$$
|u|^p \in L^1(\Omega)\,\,and\,\,|\nabla u| \in M^{\frac{2p}{\gamma +1+p}}(\Omega).
$$
\end{enumerate}
\end{thm}
Observe that the presence of the lower order term $|u|^{p-1}u$ guarantees the existence of a distributional solution
if $f$ is a $L^1(\Omega)$ function. On the contrary, problem (\ref{without}) may have no distributional solution, because the summability 
of the gradient of the solutions may be lower than 1. 

In the case where $f \in L^m(\Omega)$ with $m>1$, we will prove the following existence result.
\begin{thm}\label{thm_soluzioni_a_energia_finita}
Let $f \in L^{m}(\Omega)$, with $m>1$. 
\begin{enumerate}
\item
If $p \geq \dys \frac{\gamma +1}{m-1}$,
then there exists a distributional solution $u$ to problem (\ref{problema_di_Dirichlet}) such that 
$$
u\in H^1_0(\Omega) \cap L^{pm}(\Omega)\,.$$
\item 
If $\dys\frac{\gamma}{m-1}< p < \dys \frac{\gamma +1}{m-1}$, 
then there exists a distributional solution $u$ to problem (\ref{problema_di_Dirichlet})
such that 
$$|u|^{pm} \in L^{1}(\Omega)\,\, and\,\,u \in W^{1,\frac{2pm}{\gamma +1+p}}_0(\Omega)\,.$$ 
\item
If  $\dys 0 < p\leq \frac{\gamma}{m-1}$, then there exists an entropy solution $u$ 
to problem (\ref{problema_di_Dirichlet})
such that
$$|u|^{pm} \in L^{1}(\Omega)\,\, and\,\,  
|\nabla u| \in M^{\frac{2pm}{\gamma +1+p}}(\Omega).
$$
\end{enumerate}
\end{thm}
In this case too the lower order term $|u|^{p-1}u$ has a regularizing effect. Indeed, if $p>\frac{\gamma}{m-1}$, we have a distributional solution
for any $m>1$. This is not always the case for problem (\ref{without}). If $p$ is larger, that is 
$p>\frac{\gamma +1}{m-1}$  we even find  
a finite energy (i.e., $H^1_0(\Omega)$) solution, for any $m>1$.

Notice that in the case where $\gamma=0$, we get the same results obtained in \cite{cirmi}, for elliptic coercive problems;
moreover these results are a generalization of \cite{bb}.

\medskip
In the spirit of problem (\ref{laplaciano_asintoto}), the second problem that we will study in this paper is
\begin{equation}\label{probl_asintoto}
\left\{
\begin{array}{cl}
\displaystyle
-\diverg\left(a(x,u)\nabla u\right)+h(u)=f& {\rm in}\,\,\Omega
\\
u=0 & {\rm on}\,\,\Omega
\end{array}
\right.
\end{equation}
under the following assumptions on the lower order term $h$. Let $\sigma$ be 
a real positive constant; we assume that the function $h: [0,\sigma) \to \R$ 
is increasing, continuous, 
$h(0)=0$
and
$
\lim_{s\to \sigma^-} h(s)=+\infty
$.

We will show the existence of a  solution $u$ such that
 $h(u) \in L^1(\Omega)$ and
$$
\int_{\Omega}{a(x,u)}{\nabla u}\cdot \nabla \varphi+
\int_{\Omega}h(u)\varphi=\int_{\Omega}f\,\varphi
$$
for every $\varphi \in H^1_0(\Omega)\cap L^{\infty}(\Omega).$

\begin{thm}\label{asintoto}
Let $f$ be a positive $L^{1}(\Omega)$ function. Then there exists a solution 
$u \in H^1_0(\Omega)\cap L^{\infty}(\Omega)$ 
to problem (\ref{probl_asintoto}).
\end{thm}
This result shows that the presence of the lower order term $h(u)$ has a strong regularizing effect: it is sufficient to compare
the summability that we have found to summabilities (\ref{somm_u}) and (\ref{somm_grad}) for the solution to the same problem, 
without the lower order term $h(u)$.

Notice that the same results of Theorems \ref{dato_in_luno}, \ref{thm_soluzioni_a_energia_finita} and \ref{asintoto} can be obtained
for more general lower order terms, depending also on the variable $x$. 
\section{Preliminaries}
We will use the following approximating problems to prove
Theorems \ref{dato_in_luno} and \ref{thm_soluzioni_a_energia_finita}: 
\begin{equation}\label{problemi_approssimanti}
\left\{
\begin{array}{cl}
\displaystyle
-\diverg\left(a(x,T_n(v))\nabla v\right)+|v|^{p-1}v=T_n(f)& {\rm in}\,\,\Omega
\\
v=0 & {\rm on}\,\,\partial \Omega\,.
\end{array}
\right.
\end{equation}
The following lemma will be very useful, as it gives us an {\it a priori} 
estimate on the summability of the solutions to problems (\ref{problemi_approssimanti}). 
\begin{lemma}\label{stime_lemma}
Let $f \in L^{m}(\Omega)$, $m \geq 1$. 
Then, for every $n \in \N$, there exists a solution $u_n \in H^1_0(\Omega)$ to
problems (\ref{problemi_approssimanti}) such that
\begin{equation}\label{stima-lemma}
\int_{\Omega}|u_n|^{pm}\leq \int_{\Omega}|f|^m\,;
\end{equation}
$u_n$ satisfies
\begin{equation}\label{form_debole_pb_appross}
\int_{\Omega}a(x,T_n(u_n))\nabla u_n\cdot \nabla \varphi+
\int_{\Omega}|u_n|^{p-1}u_n\varphi=\int_{\Omega}T_n(f)\,\varphi
\end{equation}
for every $\varphi \in H^1_0(\Omega)\,.$
\end{lemma}
\begin{proof}
The operator
$$
v \to -\diverg\left(a(x,T_n(v))\nabla v\right)+|v|^{p-1}v\,,\,\,\,\,\,\,\,v \in H^1_0(\Omega)
$$ 
satisfies the hypothesis of standard theorems for elliptic operators (see \cite{L}); 
therefore there exists a solution $u_n \in H^1_0(\Omega)$ to problems (\ref{problemi_approssimanti})
which satisfies
(\ref{form_debole_pb_appross}).

\noindent
To show estimate (\ref{stima-lemma}), we will consider the cases $m>1$ e $m=1$ separately.

{\it Case $m>1$:}
Choosing $\varphi=|u_n|^{p(m-1)}sgn(u_n)$ in (\ref{form_debole_pb_appross}) 
one gets
$$
\int_{\Omega}|u_n|^{pm}
\leq 
\int_{\Omega}|f||u_n|^{p(m-1)};
$$
the H\"older inequality on the right hand side yields
$$
\int_{\Omega}|u_n|^{pm}\leq \norm{f}_{m}\left[\int_{\Omega}|u_n|^{pm}\right]^{1-\frac 1m}
$$
which implies (\ref{stima-lemma}).

{\it Case $m=1$:}
Choosing $\varphi=\dys \frac{T_k(u_n)}{k}$ in (\ref{form_debole_pb_appross})
we get
$$
\int_{\Omega}|u_n|^{p-1}u_n \frac{T_k(u_n)}{k}\leq
\int_{\Omega}T_n(f)\,\frac{T_k(u_n)}{k}\leq\int_{\Omega}|f|\,.
$$
Fatou's lemma implies, for $k\to 0$, estimate (\ref{stima-lemma}).
\end{proof}

In the study of entropy solutions to problem (\ref{problema_di_Dirichlet}) we will need 
the following lemmas.
\begin{lemma}\label{lemmabdo}
Let $u$ be a measurable function in $M^s(\Omega)$, $s>0,$ and suppose that there exists a positive constant $\rho>0$ such that 
$$
\int_\Omega|\nabla T_k(u)|^2 \leq C\,k^{\rho}\,\,\,\forall\,k>0\,,\,\,\forall\,n \in \N\,. 
$$
Then $|\nabla u| \in M^{\frac{2s}{\rho + s}}(\Omega)$.
\end{lemma}
\begin{rem}\label{remarkbdo}
Lemma \ref{lemmabdo} is true for sequences too. That is, if $u_n$ is a sequence of measurable functions such that 
$$
|\{|u_n|\geq k\}|\leq \frac{C}{k^s}\,\,\,\forall\, n\in \N
$$
$s>0,$ and there exists a positive constant $\rho>0$ such that 
$$
\int_\Omega|\nabla T_k(u_n)|^2 \leq C\,k^{\rho}\,\,\,\forall\,k>0\,,
$$
then 
$$
|\{|\nabla u_n|\geq k\}|\leq \frac{C}{k^{\frac{2s}{\rho +s}}}\,\,\,\forall\, n\in \N\,.
$$ 
\end{rem}
\begin{proof}
See \cite{bdo}.
\end{proof}
\begin{lemma}\label{lemmadeisei}
Let $u_n$ be a sequence of measurable functions such that $T_k(u_n)$ is bounded in $H^1_0(\Omega)$
for every $k>0$. Then there exists a measurable function $u$ such that $T_k(u) \in H^1_0(\Omega)$ and, up to a subsequence,
$u_n\to u$ a.e. in $\Omega$ and $T_k(u_n)\to T_k(u)$ weakly in $H^1_0(\Omega)$. 
\end{lemma}
\begin{proof}
See \cite{BBGGPV}.
\end{proof}

\section{Proofs}
To show Theorems \ref{dato_in_luno} and \ref{thm_soluzioni_a_energia_finita}, we will pass to the limit in (\ref{form_debole_pb_appross}).
All along the proofs $C$ will denote a constant depending only on $p, m, \alpha, \norm{f}_m$ and $\gamma$.
 
Let us begin with the proof of Theorem \ref{dato_in_luno}.
\begin{proof}
We divide our proof into two parts: in the first one we will study the case where $p>\gamma +1$, showing the existence 
of a distributional solution to problem (\ref{problema_di_Dirichlet}), and in the second one we will show the existence of an entropy solution for $0<p\leq \gamma +1$.

{\it PART I:} Let $p>\gamma +1$.
If we choose $\varphi=[(1+|u_n|)^{1-\lambda}-1]sgn(u_n)$, with $\lambda >1$, in (\ref{form_debole_pb_appross}) we get,
using assumption (\ref{ellitticita}) on $a$ 
\begin{equation}
C\int\limits_{\Omega}\frac{|\nabla u_n|^2}{(1+|u_n|)^{\gamma+\lambda}}
\leq \int\limits_{\Omega}|f|\,.
\label{stima_trucco}
\end{equation}
Now, let $q<2$;
 writing
$$
\int_{\Omega}|\nabla u_n|^{q}=
\int_{\Omega}\frac{|\nabla u_n|^{q}}{(1+|u_n|)^{\frac{q}{2}(\gamma+\lambda)}}(1+|u_n|)^{\frac{q}{2}(\gamma+\lambda)}
$$
and
using the H\"older inequality with exponent $\frac{2}{q}$, we get from (\ref{stima_trucco})
$$
\int_{\Omega}|\nabla u_n|^{q}\leq C\left(\int_{\Omega}(1+|u_n|)^{\frac{q}{2-q}(\gamma +\lambda)} \right)^{1-\frac q2}\,.
$$
Thanks to Lemma \ref{stime_lemma}, the right hand side is uniformly bounded if $\frac{q}{2-q}(\gamma + \lambda)=p$, 
that is $q=\frac{2p}{\gamma + p+ \lambda}$. 
Since $\lambda>1$, then $q>1$ and so we get that $u_n$ is uniformly bounded in $W^{1,\beta}_0(\Omega)$, for
$\beta<\frac{2p}{\gamma +p+1}$.
As a consequence there exists a function $u \in W^{1,\beta}_0(\Omega)$, $\beta<\frac{2p}{\gamma +p+1}$ such that
$u_n\to u$ weakly in $W^{1,\beta}_0(\Omega)$, 
up to a subsequence. Moreover $u_n \to u$ a.e. in $\Omega$ and this implies that $u \in L^p(\Omega)$.

We are going to show that $u$ is a distributional solution to 
problem (\ref{problema_di_Dirichlet})
passing to the limit in (\ref{form_debole_pb_appross}). We will suppose that $\varphi \in C^{\infty}_0(\Omega).$
For the first term of the left-hand side, it is sufficient to observe that 
$\nabla u_n \to \nabla u$ weakly in $L^\beta(\Omega)$ and 
$a(x,T_n(u_n)){\nabla \varphi}\to a(x,u){\nabla \varphi}$ 
in $L^m(\Omega)$ for every $m$, due to assumption (\ref{crescita}) on $a$.
The limit of the second term is a little more delicate.
Even if the principal part of our differential operator is degenerate, on can show the following estimate, in the spirit of  \cite{BG1}: 
\begin{equation}\label{boccardo-gallouet}
\int\limits_{\{|u_n|>t\}}|u_n|^p
\leq 
\int\limits_{\{|u_n|>t\}}|f|\,.
\end{equation}
Indeed, let 
$\psi_i$ be a sequence of  increasing, positive, uniformly bounded $C^{\infty}(\Omega)$ functions, 
 such that 
$$
\psi_i(s)\to \left\{
\begin{array}{ll}
1,& s\geq t
\\
0, & |s|<t
\\
-1, & s\leq -t\,.
\end{array}
\right.
$$
Choosing $\psi_i(u_n)$ in (\ref{form_debole_pb_appross}), we get
$$
\int\limits_{\Omega}|u_n|^{p-1}u_n\,\psi_i(u_n)\leq 
\int\limits_{\Omega}T_n(f)\,\psi_i(u_n)\,.
$$
The limit on $i$ implies (\ref{boccardo-gallouet}). 
We are going to use this inequality to show that if 
$E$ is any measurable subset of $\Omega$, then
$$
\lim\limits_{|E|\to 0}\int\limits_E |u_n|^p=0\,\,\,\,\mbox{uniformly with respect to}\,\,n\,.
$$
Using assumption (\ref{boccardo-gallouet}), for any $t>0$ we have
$$
\int\limits_E |u_n|^p\leq t^p\, |E|+\int\limits_{E\cap \{|u_n|>t\}} |u_n|^p\leq t^p\, |E|+
\int\limits_{\{|u_n|>t\}} |f|\,.
$$
The previous lemma and the fact that $f \in L^1(\Omega)$ allow us to say that
for any given $\varepsilon>0$, there exists $t_{\varepsilon}$ such that
$$
\int\limits_{\{|u_n|>t_{\varepsilon}\}} |f|\leq \varepsilon\,.
$$
In this way
$$
\int\limits_E |u_n|^p\leq t_{\varepsilon}^p\, |E|+\varepsilon 
$$
and so
$$
\lim\limits_{|E|\to 0}\int\limits_E |u_n|^p\leq \varepsilon \,\,\,\,\,\,\,\forall\,\,\varepsilon>0\,.
$$ 
We thus proved that
$
\lim\limits_{|E|\to 0}\int\limits_E |u_n|^p=0$ uniformly with respect to $n$. 
Vitali's theorem implies that $|u_n|^{p-1}u_n \to |u|^{p-1}u$ in $L^1(\Omega)$ 
and so we can pass to the limit in
the second term  of the left-hand side of (\ref{form_debole_pb_appross}). 
We get a distributional solution to problem (\ref{problema_di_Dirichlet}) in this way.

\smallskip
{\it Part II:}
Let $p>0$. 
Let us choose 
$T_k(u_n)$ as test function in (\ref{form_debole_pb_appross}):
$$
\int_{\Omega} a(x, T_n(u_n))\nabla u_n\cdot \nabla T_k(u_n)+\int_{\Omega}|u_n|^{p-1}u_n\,T_k(u_n)
=
\int_{\Omega} T_n(f)\,T_k(u_n)\,.
$$
In this way we have
$$
\int_{\Omega}a(x, T_n(u_n))\nabla u_n\cdot \nabla T_k(u_n)\leq C\,k\,,
$$
and so
$$
\int_{\Omega}|\nabla T_k(u_n)|^2 \leq C k(1+k)^{\gamma}\,.
$$
Lemma \ref{lemmadeisei} implies the existence of a measurable function $u$ such that $\forall\, k>0$,  $T_k(u) \in H^1_0(\Omega)$, 
and, up to a subsequence, $T_k(u_n)\to T_k(u)$ weakly in $H^1_0(\Omega)$, and  
$u_n \to u$ a.e.  in $\Omega$.
Fatou's lemma implies that $|u|^p \in L^1(\Omega)$.  
To prove that $|\nabla u| \in M^{\frac{2p}{\gamma +p+1}}(\Omega)$, one can pass
 to the limit, as $n\to \infty,$ in the previous inequality
to get
$$
\int_{\Omega}{|\nabla T_k(u)|^2}\leq 
 C\,k\,{(1+k)^{\gamma}}\,.
$$
Using Lemma \ref{lemmabdo},  $|\nabla u| \in M^{\frac{2p}{\gamma +1+p}}(\Omega)$.

We claim that $u$ is an entropy solution. 
 Indeed, let us choose 
$T_k(u_n-\varphi)$, $\varphi \in H^1_0(\Omega)\cap L^{\infty}(\Omega)$  as test function in (\ref{form_debole_pb_appross}):
$$
\int_{\Omega} a(x, T_n(u_n))\nabla u_n\cdot \nabla T_k(u_n-\varphi)+\int_{\Omega}|u_n|^{p-1}u_n\,T_k(u_n-\varphi)
=
\int_{\Omega} T_n(f)\,T_k(u_n-\varphi)\,.
$$
For the second term of the left hand side and for the right hand side, one can use the Fatou's lemma to pass to the limit. 
For the first term of the left hand side, we can use the same technique as in \cite{bdo} (Theorem 1.17). 
Indeed, let us write it as
$$
\int_\Omega a(x,T_n(u_n))|\nabla T_k(u_n-\varphi)|^2 + 
\int_\Omega a(x,T_n(u_n))\nabla \varphi\cdot \nabla T_k(u_n-\varphi)\,.
$$
Passing to the limit in the first term, we have, since $T_k(u_n) \to T_k(u)$ weakly in $H^1_0(\Omega)$ and $u_n \to u$ a.e. in $\Omega$,
$$
\liminf_{n\to \infty}\int_\Omega a(x,T_n(u_n))|\nabla T_k(u_n-\varphi)|^2 \geq
\int_\Omega a(x,u)|\nabla T_k(u-\varphi)|^2\,;
$$
the second term tends to
$\dys\int_\Omega a(x,u)\nabla \varphi\cdot \nabla T_k(u-\varphi)$\,.
Therefore 
$$
\liminf_{n\to \infty}\int_\Omega a(x, T_n(u_n))\nabla u_n\cdot \nabla T_k(u_n-\varphi)\geq \int_\Omega a(x, u)\nabla u\cdot \nabla T_k(u-\varphi)\,.
$$
\end{proof}

Let us show Theorem \ref{thm_soluzioni_a_energia_finita}, that is the existence of solutions to problem 
(\ref{problema_di_Dirichlet}) in the case where $f \in L^m(\Omega)$, $m>1$.
\begin{proof}
We will divide our proof into three parts, according to the different values of $p$.

{\it PART I}: Suppose  $\dys p\geq \frac{\gamma+1}{m-1}$.
If we choose $\varphi=[(1+|u_n|)^{\gamma+1}-1]sgn(u_n)$ in (\ref{form_debole_pb_appross}) we get,
using assumption (\ref{ellitticita}) on $a$ 
$$
C\int\limits_{\Omega}\frac{|\nabla u_n|^2}{(1+|T_n(u_n)|)^{\gamma}}(1+|T_n(u_n)|)^{\gamma}
\leq \int\limits_{\Omega}|T_n(f)|\,[(1+|u_n|)^{\gamma+1}-1]\,.
$$
As a consequence
$$
\int\limits_{\Omega}|\nabla u_n|^2 
\leq
C\int\limits_{\Omega}|f||u_n|^{\gamma + 1}\,.
$$
The H\"older inequality on the right hand side and Lemma \ref{stime_lemma} imply that
$$
\int\limits_{\Omega}|f||u_n|^{\gamma + 1}\leq 
\norm{f}_m\left[\int\limits_{\Omega}|u_n|^{(\gamma + 1)\frac{m}{m-1}}\right]^{\frac{m-1}{m}}
<\infty
$$
if $(\gamma+1)\frac{m}{m-1}\leq pm$. Under this assumption
$$
\int\limits_{\Omega}|\nabla u_n|^2 \leq C\,,\,\,\,\,\forall\,n\in \N\,.
$$
Up to a subsequence, 
there exists a function $u \in H^1_0(\Omega)$ such that
$u_n\to u$ weakly in $H^1_0(\Omega)$ and a.e. in $\Omega$. Moreover $u \in L^{pm}(\Omega)$.

We are going to show that $u$ is a solution to problem (\ref{problema_di_Dirichlet}), 
passing to the limit in (\ref{form_debole_pb_appross}). 
For the first term of the left-hand side, it is sufficient to observe that 
$
a(x,T_n(u_n)){\nabla \varphi}\to a(x,u){\nabla \varphi}$ in $L^r(\Omega)$, for any $r\geq 1$,
due to assumption (\ref{crescita}) on $a$.
For the second term, since $|u_n|^{p-1}u_n$ is uniformly bounded in $L^m(\Omega)$ with $m>1$ and $u_n \to u$ a.e., we can deduce that
$|u_n|^{p-1}u_n\to |u|^{p-1}u$ in $L^1(\Omega)$.

\smallskip
{\it PART II:}
Suppose $\dys \frac{\gamma}{m-1}<p< \frac{\gamma+1}{m-1}.$
If we choose $\varphi=[(1+|u_n|)^{p(m-1)}-1]sgn(u_n)$ in (\ref{form_debole_pb_appross}), we get,
using assumption (\ref{ellitticita}) on $a$ 
$$
\int\limits_{\Omega}\frac{|\nabla u_n|^2}{(1+|u_n|)^{\gamma-p(m-1)+1}}
\leq C \int\limits_{\Omega}|f|\,|u_n|^{p(m-1)}\,.
$$
Now, using the H\"older inequality in the right hand side of the previous inequality and Lemma \ref{stime_lemma}, we get
\begin{equation}
\int\limits_{\Omega}\frac{|\nabla u_n|^2}{(1+|u_n|)^{\gamma-p(m-1)+1}}
\leq C\left[\int\limits_{\Omega}|u_n|^{pm}\right]^{1-\frac{1}{m}}\leq C \,\,\,\,\forall\, n\in \N\,.
\label{stima_trucco_m_maggiore_uno}
\end{equation}
From the other hand, let $\sigma<2$;
 writing
$$
\int\limits_{\Omega}|\nabla u_n|^{\sigma}=
\int\limits_{\Omega}\frac{|\nabla u_n|^{\sigma}}{(1+|u_n|)^{\frac{\sigma}{2}[\gamma-p(m-1)+1]}}(1+|u_n|)^{\frac{\sigma}{2}[\gamma-p(m-1)+1]}
$$
and
using the H\"older inequality with exponent $\frac{2}{\sigma}$, we get from (\ref{stima_trucco_m_maggiore_uno})
$$
\int\limits_{\Omega}|\nabla u_n|^{\sigma}\leq C\left(\int_{\Omega}(1+|u_n|)^{\frac{\sigma}{2-\sigma}[\gamma -p(m-1)+1]} \right)^{1-\frac \sigma2}\,.
$$
Thanks to Lemma \ref{stime_lemma},  if $\dys \frac{\sigma}{2-\sigma}[\gamma-p(m-1)+1]= pm$, that is $\dys \sigma=\frac{2pm}{\gamma + p+ 1}$, the last quantity is uniformly bounded. 
Notice that $\sigma<2$,  since we are assuming that $\dys p<\frac{\gamma +1}{m-1}$.
If $\dys \frac{2pm}{\gamma + p+ 1}>1$, the fact that
$$
\int\limits_{\Omega}|\nabla u_n|^{\frac{2pm}{\gamma +p+1}}\leq C\,\,\,\forall\, n\in \N
$$  
implies the existence of a function $u \in W^{1,\frac{2pm}{\gamma +p+1}}_0(\Omega)$ such that, up to a subsequence, 
$u_n\to u$ weakly in $W^{1,\frac{2pm}{\gamma +p+1}}_0(\Omega)$ and a.e. in $\Omega$; moreover $|u|^{pm}\in L^1(\Omega)$.
One can show that $u$ is a distributional solution to 
problem (\ref{problema_di_Dirichlet}), as in {\it PART I}.

\smallskip
{\it PART III:}
Suppose that $\dys p\leq \frac{\gamma}{m-1}$. We show
that there exists an entropy solution to problem (\ref{problema_di_Dirichlet}). 
Estimate (\ref{stima_trucco_m_maggiore_uno}) (obtained indipendently from the value of $p$) implies that
$$
\int\limits_{\{|u_n|>k\}}\frac{|\nabla u_n|^2}{(1+|u_n|)^{\gamma-p(m-1)+1}}
\leq C
$$
and consequently
$$
\int\limits_{\{|u_n|<k\}}|\nabla T_k(u_n)|^2\leq C(1+k)^{\gamma-p(m-1)+1}\,.
$$
Thanks to Lemma \ref{lemmadeisei}, there exists a function $u$ such that 
$T_k(u) \in H^1_0(\Omega), \,\forall\, k$, $T_k(u_n)\to T_k(u)$ weakly in $H^1_0(\Omega)$ and $u_n \to u$ a.e. in $\Omega$.
Therefore we can pass to the limit as $n \to \infty$ in the previous inequality to get
$$
\int_{\Omega}|\nabla T_k(u)|^2\leq C(1+k)^{\gamma-p(m-1)+1}\,.
$$
Lemma \ref{lemmabdo} gives us, if $p<\frac{\gamma +1}{m-1}$ 
$$
|\nabla u| \in M^{\frac{2pm}{\gamma +1+p}}(\Omega)\,.
$$
Since $|u_n|^{pm}$ is uniformly bounded in $L^1(\Omega)$, we can say that $|u|^{pm} \in L^1(\Omega)$.
One can show that $u$ is an entropy solution, using the same method as in {\it PART II} of the proof of Theorem \ref{dato_in_luno}. 
\end{proof}  

We are going to show Theorem \ref{asintoto}. We will follow the technique used in \cite{B}.
In the lemma below we prove the existence of a solution to problem (\ref{probl_asintoto}) in the case where the datum $f$ is a bounded function.
\begin{lemma}
Let $f$ be a positive $L^{\infty}(\Omega)$ function. Then there exists a distributional solution $u \in L^{\infty}(\Omega)\cap H^1_0(\Omega)$
to problem (\ref{probl_asintoto}). Moreover $0\leq u(x)\leq \sigma$ a.e. in $\Omega$.
\end{lemma}
\begin{proof}
The standard theorems on elliptic operators (see \cite{L}) imply the existence of a solution $u_n \in H^1_0(\Omega)$ to the approximating problems
 \[
\left\{
\begin{array}{cl}
\displaystyle
-{\rm div}\left(a(x,T_n(v)){\nabla v}\right)+h_n(v)=f& {\rm in}\,\,\Omega
\\
v=0 & {\rm on}\,\,\Omega
\end{array}
\right.
\]
where 
$$
h_n(s)=
\left\{
\begin{array}{cl}
h(s), & h(s)<n \,\,{\textnormal{and}}\,\, s<\sigma
\\
n, & h(s)\geq n \,\,{\textnormal{and}}\,\, s<\sigma
\\
n, & s\geq \sigma\,.
\end{array}
\right.
$$
The use of $ (u_n  - h^{-1}(\norm{f}_{\infty}))^+$  as test
function yields
$$
  \int_{\Omega} (h_n(u_n) - f) (u_n  - h^{-1}(\norm{f}_{\infty}))^+\leq 0.
$$
 From the  previous inequality we have
$$
0\geq \int\limits_{\{h(u_n)  \geq  \norm{f}_{\infty}\}}
(T_n(h(u_n)) - f) (u_n  -h^{-1}(\norm{f}_{\infty})) 
$$
$$
=
\int\limits_{\{ n> h(u_n)  \geq   \norm{f}_{\infty}\}}
( h(u_n)  - f) (u_n  -h^{-1}(\norm{f}_{\infty}))
$$
$$
+
\int\limits_{\{h(u_n)  \geq n\geq  \norm{f}_{\infty}\}}
(n - f) (u_n  -h^{-1}(\norm{f}_{\infty}))
\geq 0.
$$
This implies that
$$
0\leq h(u_n) \leq \norm{f}_{\infty}\,,
$$
and so 
\begin{equation}
\label{stima_infinito}
0\leq u_n \leq h^{-1}(\norm{f}_{\infty})\leq \sigma-\varepsilon\,.
\end{equation}
Choosing $\varphi=u_n$ we get an uniform bound 
for the $H^1_0(\Omega)$ norm of $u_n$. Indeed using assumption (\ref{ellitticita}) on $a$ 
and the fact that
$h(u_n)\geq 0$, (\ref{stima_infinito}) implies that 
$$
\alpha \int\limits_{\Omega}\frac{|\nabla u_n|^2}{(1+\sigma)^{\gamma}}\leq \norm{f}_1\sigma\,.
$$
This yields the existence of a function $u\in H^1_0(\Omega)\cap L^{\infty}(\Omega)$ 
such that, up to a subsequence, $u_n\to u$ weakly in $H^1_0(\Omega)$ and a.e. 
in $\Omega$. Remark that $u_n=T_n(u_n)$ for $n$ large enough. Consequently 
$u_n$ satisfies 
$$
\int\limits_{\Omega}a(x,u_n)\nabla u_n\cdot \nabla \varphi+
\int\limits_{\Omega}h_n(u_n)\varphi=\int\limits_{\Omega}f\,
\varphi\quad \forall\,\varphi\in H^1_0(\Omega)\cap L^{\infty}(\Omega)\,.
$$
Let us show that $u$ is a solution to (\ref{probl_asintoto}), 
passing to the limit in the previous equality. About the first term it is sufficient to observe that
$a(x,u_n)\nabla u_n \to a(x,u)\nabla u$ weakly in $L^2(\Omega)$, because $a(x,u_n)\to a(x,u)$ in $L^2(\Omega)$ and 
$\nabla u_n \to \nabla u$ weakly in $L^2(\Omega)$. 
One can use the same technique as in \cite{B} to pass to the limit in the second term using (\ref{stima_infinito}) and to show that $h(u) \in L^1(\Omega)$.  
\end{proof}
We remark that the previous lemma gives us an upper bound for the solution $u$ (that is $u\leq \sigma-\varepsilon$) independently on the datum $f$.

We are now able to show Theorem \ref{asintoto}.
\begin{proof}
The following sequence of approximating problems
$$
\left\{
\begin{array}{cl}
\displaystyle
-\diverg\left(a(x,v)\nabla v\right)+h(v)=T_n(f)& {\rm in}\,\,\Omega
\\
v=0 & {\rm on}\,\,\partial \Omega
\end{array}
\right.
$$
has a solution
$u_n$ such that $0\leq u_n\leq \sigma-\varepsilon$ 
for every $n\in \N$, according to the previous lemma; $u_n$ satisfies
\begin{equation}\label{asintoto-teorema}
\int\limits_{\Omega}a(x,u_n)\nabla u_n\cdot \nabla \varphi+
\int\limits_{\Omega}h_n(u_n)\varphi=\int\limits_{\Omega}T_n(f)\,\varphi\quad \forall\,\varphi\in H^1_0(\Omega)\cap L^{\infty}(\Omega)\,.
\end{equation}
As in the previous lemma one has
$$
\alpha \int\limits_{\Omega}\frac{|\nabla u_n|^2}{(1+\sigma)^{\gamma}}\leq ||f||_1\,\sigma\,.
$$
Therefore there exists a function $u \in H^1_0(\Omega)$ such that, up to a subsequence,
$u_n\to u$ weakly in $H^1_0(\Omega)$ and a.e. in $\Omega$; moreover $0\leq u \leq \sigma$. 
To show that $u$ is a solution to (\ref{probl_asintoto}), 
we pass to the limit in (\ref{asintoto-teorema}). 
As in the previous Lemma, it is easy to prove that
$a(x,u_n)\nabla u_n \to a(x,u)\nabla u$ weakly in $L^2(\Omega)$. 
One can use the same technique as in \cite{B} to pass to the limit in the second term. Indeed 
showing the estimate
$$
\int\limits_{\{s\leq u_n< \sigma\}} h(u_n) \leq \int\limits_{\{s\leq u_n< \sigma\}} f 
$$
one can prove that 
for any measurable set $E\subset \Omega$
$$
\int_E h(u_n)\leq \int\limits_{\{s\leq u_n< \sigma\}} f + h(s)|E|\,.
$$
Since $\lim\limits_{s \to \sigma}|\{s\leq u_n< \sigma\}|=0$ we get that
$\{h(u_n)\}$ is equi-integrable.  This implies that $h(u_n)\to h(u)$ in $L^1(\Omega),$ as required.
\end{proof}



\end{document}